\documentclass{ws-procs9x6}

\usepackage{amsmath, amssymb,xspace}

\usepackage{graphicx}

\newcommand{\cost}{\mathbf c}

\newcommand{\DII}{\Delta^0_2} 
\newcommand{\NN}{{\mathbb{N}}} 
\newcommand{\RR}{{\mathbb{R}}} 
\newcommand{\QQ}{{\mathbb{Q}}} 
 
\newcommand{\sub}{\subseteq} 
\newcommand{\sN}[1]{_{#1\in \NN}} 
\newcommand{\uhr}[1]{\! \upharpoonright_{#1}} 
\newcommand{\ML}{Martin-L{\"o}f} 
\newcommand{\SI}[1]{\Sigma^0_{#1}} 
\newcommand{\PI}[1]{\Pi^0_{#1}} 
\newcommand{\PPI}{\PI{1}}

\newcommand{\bc}{ \begin{center}} 
\newcommand{\ec}{ \end{center}}

\newcommand{\Halt}{{\ES'}} 
\newcommand{\ES}{\emptyset}

\newcommand{\tp}[1]{2^{#1}}

\newcommand{\la}{\langle} 
\newcommand{\ra}{\rangle}

\newcommand{\Kuc}{Ku{\v c}era{}}

\newcommand{\fao}[1]{\forall #1 \, } 
\newcommand{\exo}[1]{\exists #1 \, } 
 
\newcommand{\leT}{\le_{\mathrm{T}}} 
 
\newcommand{\ltt}{\le_{\mathrm{tt}}}

\newcommand{\Om}{\Omega} 
\newcommand{\n}{ 
\noindent} 
 
\newcommand{\vsps}{\vspace{3pt}} 
 
\newcommand{\vsp}{\vspace{6pt}} 
\newcommand{\leb}{\mathbf{\lambda}}

\newcommand{\sss}{\sigma}

\newcommand{\lland}{\ \land \ }

\newcommand{\seq}[1]{\la #1 \ra }

\newcommand\+[1]{\mathcal{#1}}

\newcommand{\UM}{\mathbb{U}}

\newcommand{\LR}{\Leftrightarrow} 
 
\newcommand{\LLR}{\ \Leftrightarrow \ } 
\newcommand{\RA}{\Rightarrow} 
\newcommand{\RRA}{\ \Rightarrow\ } 
\newcommand{\LA}{\Leftarrow}

\begin{document}

\title{Calibrating the complexity of $\Delta^0_2$ sets  \\ 
via their changes}

\author{Andr\'e Nies} 

\address{Dept.\ of Computer Science, University of Auckland}

\begin{abstract}  The  computational complexity of a $\DII$ set will be calibrated by  the amount of changes needed for  any of its computable approximations. Firstly,   we study \ML\ random sets, where we quantify the changes of initial segments. Secondly, we look at c.e.\ sets, where we quantify the overall   amount of changes   by obedience to cost functions.  Finally, we combine the two settings. The discussions lead to three basic principles on  how complexity and changes relate. 
\end{abstract}


\thanks{Supported by the Marsden fund of New Zealand} 

\bodymatter

\section*{Introduction} \setcounter{section}{1} In computability theory one studies the complexity of sets of natural numbers. A good arena for this is the class of~$\DII$ sets, that is, the sets Turing below the Halting problem $\Halt$. For, by the Shoenfield Limit Lemma, they can be approximated in a computable way. More precisely, the lemma says that a set $Z\sub \NN$ is Turing below the halting problem $\Halt$ if and only if there is a computable function $g \colon \, \NN \times \NN \to \{0,1\}$ such that {$Z(x)= \lim_s g(x,s)$} for each $x\in \NN$. We will write {$Z_s$} for $\{x \colon \, g(x,s)=1\}$. The sequence $\seq {Z_s} \sN s$ is called a \emph{computable approximation} of $Z$. 

The paper is set up as a play in three acts. The main topic of the play is to study the complexity of a $\DII$ set $Z$ by quantifying the amount of changes that are needed in any computable approximation $(Z_s)\sN s$ of $Z$. \begin{itemize}
\item In the first act, we will do this for random $\DII$ sets. They are played by knights living in a castle who do a lot of horseback riding. 
\item In the second act, we will do it mainly for computably enumerable (c.e.) sets. They are played by poor peasants living in a village who are trying  to pay their taxes. 
\item In the final act, we will relate the two cases. The knights and the peasants meet. \end{itemize}

The purpose of this  work is to  provide a unifying background for  results in   the papers  \citelow{Downey.Greenberg:nd,Figueira.Hirschfeldt.ea:nd,Figueira.ea:08,Franklin.Ng:10,Greenberg.Hirschfeldt.ea:nd,Greenberg.Nies:11,Kucera:86}. It contains many new observations on the amount of changes of knights and peasants, and how they relate. However, it  does not contain new technical results. 

\subsection*{\ML\ randomness} Our central algorithmic randomness notion is the one due to \ML\  \citelow{Martin-Lof:66}. It has many equivalent definitions. We give one: 
\begin{definition} \label{def:MLR}
	We say that a set $Z\sub \NN$ is \ML\ random (ML-random) if for every computable sequence $(\sss_i)\sN i$ of binary strings with $\sum_i \tp{-|\sss_i|} < \infty$, there are only finitely many $i$ such that $ \sss_i$ is an initial segment of $Z$. 
\end{definition}
Note that $\lim_i \tp{-|\sss_i|} = 0$, so this means that we cannot ``Vitali cover'' $Z$ (viewed as the binary expansion of a real number) with the collection of dyadic intervals corresponding to $(\sss_i)\sN i$. A sequence $(\sss_i)\sN i$ as above is called a Solovay test (see e.g.\ \cite[3.2.2]{Nies:book}).

\subsection*{Left-c.e.\ sets} We will often consider a special type of $\DII$ set. We say that $Z\sub \NN$ is \emph{left-c.e.} if it has a computable approximation $(Z_s) \sN s$ such that $Z_s \le_L Z_{s+1} $, where $\le_L$ denotes the lexicographical ordering. For instance, $\Omega$, the halting probability of a universal prefix-free machine $\UM$ (see, for instance, \cite[Ch.\ 2]{Nies:book}), is left-c.e. To see this, let $\Omega_s$ be the measure of $\UM$-descriptions $\sss$ where the computation $\UM(\sss)$ has converged by stage $s$. This is a dyadic rational, which we identify with a binary string.

 It is well known that $\Om$ is ML-random. For general background on algorithmic randomness, see \citelow{Downey.Hirschfeldt:book,Nies:book}. 

\subsection*{Quantifying changes} We introduce the terminology needed to quantify the changes of initial segments for a computable approximation of a $\DII$ set. 
\begin{definition}
	\label{def:g_change_set} Let $g\colon \NN \to \NN$. We say that a $\DII$ set $Z$ is a {$g$-change set} if it has a computable approximation $(Z_s)\sN s$ such that an initial segment $Z_s\uhr n$ changes at most $g(n)$ times. 
	
	We also say that $Z$ is $g$-computably approximable, or \emph{$g$-c.a.} To be $\omega$-c.a.\ means to be $g$-c.a.\ for some computable $g$. 
\end{definition}
We give an important example. 
\begin{proposition}
	\label{fact:leftce_few_changes} Every left-c.e.\ set is a $g$-change set for some $g = o(2^n)$. 
\end{proposition}
\begin{proof} Fix a computable approximation $(Z_s) \sN s$ of $Z$  such that $Z_s \le_L Z_{s+1} $ for each $s$. 
	It suffices to note that if $Z \uhr k$ is stable by stage $t$, then for every $n > t$, $ Z\uhr n$ changes at most $t+ 2^{n-k}$ times.
\end{proof}
If we say that a $\DII$ set \emph{needs more than $g$ changes}, we simply mean that it is not a $g$-change set. Figueira, Hirschfeldt, Miller, Ng and Nies~\citelow{Figueira.Hirschfeldt.ea:nd} studied such lower bounds for the changes of random $\DII$ sets. 

\begin{proposition}
	\citelow{Figueira.Hirschfeldt.ea:nd} Let $Z$ be a random $\DII$ set. Let $q:\NN\to\RR^+$ be computable and nonincreasing. If $Z$ is a $\lfloor q(n) 2^n \rfloor $-change set then $\lim_n q(n)>0$. 
\end{proposition}
For example, let $q(n)= 1/ \log \log n$. Then $\lim_n q(n) = 0$. Thus, no \ML\ random set is a $\lfloor 2^n/\log \log n \rfloor $-change set. As a consequence, for the number of initial segment changes for $\Omega$, the upper bound~$o(2^n)$ is not far below $2^n$. 

\section*{Act 1: \ML\ random sets and   initial segments} \setcounter{section}{2}


\begin{minipage}
	{12cm} {\it 
\emph{The players: }

\vsps

$\Omega$, the king.

$Z$, a raundon $\DII$-knight.

More knights.

\vsps

\emph{The scene:} The fields outside a castle.

} \end{minipage}

\vsp

\subsection{Randomness enhancement} The \emph{randomness enhancement thesis} states that for a \ML\ raundon~$Z$, 

\bc \emph{$Z$ gets more random $\LR$ $Z$ is computationally less complex.} \ec The thesis was explicitly and in full generality first mentioned in  Section 4 of Nies~\citelow{Nies:tut}, and published in \citelow{Nies:5questions}. Particular instances were given  in the literature much earlier on, possibly as far back as Kurtz \citelow{Kurtz:81}. 

The thesis was was initially observed only for randomness notions not compatible with being $\DII$. Recall that a set  $Z$ is weakly 2-random if $Z$ is not in any null $\PI 2$ class; $Z$ is 2-random if it is ML-random relative to $\Halt$; $Z$ is low for $\Om$   if $\Om$ is ML-random relative to $Z$.  Lowness for $\Om$  was first studied in Nies, Terwijn and Stephan~\citelow{Nies.Stephan.ea:05}. 
\begin{example}
	Let $Z \sub \NN$ be ML-random. Then 
	\begin{itemize}
		\item[] $Z$ and $\Halt$ form a minimal pair $\LLR$ $Z$ is weakly 2-random,  
		\item[] $Z$ is low for $\Om$ $\LLR$ $Z$ is 2-random. 
	\end{itemize}
\end{example}
The first example is due to Hirschfeldt and Miller; see \cite[5.3.16]{Nies:book}. The second example follows from literature results:  by Kurtz \citelow{Kurtz:81}, 2-randomness is equivalent to randomness relative to $\ES'$. Chaitin realized that $\Om \equiv_T \ES'$. Since $\Om$ is ML-random, we can now invoke van Lambalgen's theorem to conclude that $\Om $ is ML-random in $Z$ iff $Z$ is ML-random in $\Om$.   The result  was  first explicitly  mentioned   in~\citelow{Nies.Stephan.ea:05}.

In contrast, the following, later result of Franklin and Ng~\citelow{Franklin.Ng:10} is also relevant for $\DII$ ML-random sets $Z$. A \emph{difference test} consists of a sequence of uniformly given $\SI 1 $ classes $\+ A_m$ and a further $\SI 1$ class $\+ B$ such that $\leb(\+ A_m-\+ B)\le \tp{-m}$ for each $m$. To \emph{pass} the test means to be out of $\+ A_m-\+ B$ for some~$m$. A set $Y$ is \emph{difference random} if it passes all difference tests. 

\begin{example}
	[\citelow{Franklin.Ng:10}] Let $Z$ be ML-random. Then \begin{itemize} 
	\item[] $Z$ is Turing incomplete {$\LLR$} $Z$ is difference random. \end{itemize}
\end{example}
Weak Demuth randomness is a property strictly in between weak 2-random and ML-random; see for instance \citelow{Kucera.Nies:11,Kucera.Nies:12}. Franklin and Ng have recently introduced a property of a c.e.\ set $A$ called strong promptness, which strictly  implies being promptly simple: there is a computable enumeration $(A_s)\sN s$ of $A$ and an $\omega$-c.a.\  bound $g$ such that $|W_e| \ge g(e)$   implies that $A$ promptly enumerates some element of~$W_e$. They used this property  to provide a further, related, example of randomness enhancement that is also relevant to $\DII$ sets: a ML-random $Z$ does not compute a strongly prompt set if and only if it is weakly Demuth random.

\subsection{Malory's thesis}

All  the quotes below  are from Le Morte D'Arthur (1483) by Sir Thomas Malory\footnote{Sir Thomas  Malory was an  English writer who died 1471. His major work, ``Le Morte d'Arthur'', is a prose translation of a collection of legends about King Arthur (OED). It was printed in 1483 by William Caxton, who also acted as a (somwehat sloppy) editor.}.

\n  {Book III, Chapter IX: How Sir Tor rode after the knight with the brachet\footnote{\n A brachet is a small hunting dog.}, and of his adventure by the way.} 
\begin{quote}
	(...) And anon the knight yielded him to his mercy. But, sir, I have a fellow in yonder pavilion that will have ado with you anon. He shall be welcome, said Sir Tor. Then was he ware of another knight coming with great raundon\footnote{The Old French noun ``randon'', great speed, is derived from ``randir'', to gallop. It has been used in English since the 14th century. When used in a metaphorical way, ``randon'' meant ``impetuousity'' (OED). Malory's spelling  ``raundon'' may have been  an attempt to represent the French pronounciation.}, and each of them dressed to other, that marvel it was to see; but the knight smote Sir Tor a great stroke in midst of the shield that his spear all to-shivered. And Sir Tor smote him through the shield below of the shield that it went through the cost of the knight, but the stroke slew him not. (...) 
\end{quote}

\vsps 

From this  quote   one can derive what we will call {Sir Thomas Malory's thesis}.

\vsps

\emph{Let $Z$ be a \ML\ raundon $\DII$ set. Then}

\bc \emph{$Z$ gets more raundon $\LR$ $Z$ needs more changes. } \ec 

\subsection*{Combining the two theses} We combine the randomness enhancement thesis with Malory's thesis by ``transitivity''. This yields the main principle of this act: for a ML-random $\DII$ set $Z$, 

\vsps \bc \fbox{$Z$ is computationally less complex {$\LLR$} $Z$ needs more changes.} \ec

\vsps

We will give multiple evidence for this principle. Firstly, we consider random $\DII$ sets that are complex. This should mean that they can be computably approximated with \emph{few} changes. Thereafter, we consider random $\DII$ sets that are not complex. They should need \emph{a lot} of changes.

\vsp

\n \emph{Evidence for the main principle: Complex random $\DII$ sets. }

\n 1. Chaitin's halting probability $\Omega$ is Turing complete. By Fact~\ref{fact:leftce_few_changes}, its rate of change is $o(2^n)$, which is at the bottom of the scale of possible changes for a random $\DII$ set. 

\n 2. Consider all the ML-random sets that are $\omega$-c.a.\ (Def.\ \ref{def:g_change_set}). These sets change much less than a general~$\DII$ set. By the already mentioned unpublished work of Hirschfeldt and Miller (see \cite[5.3.15]{Nies:book}), it turns out that they are ``jointly'' complex: there is an incomputable c.e.\ set Turing (even weak truth-table) below all of them. In contrast, by the low basis theorem with upper cone avoidance, for each incomputable c.e.\ set $A$, there is a ML-random $\DII$ set $Z$ not Turing above~$A$. The closer to computable $A$ is, the more $Z$ has to change; certainly $Z$ is not $\omega$-c.a.\ in general. 

\vsp 

\n \emph{Evidence for the main principle: Non-complex random $\DII$ sets.} 

\n Recall that a set $Z\sub \NN$ is \emph{low} if $Z' \leT \ES'$, and \emph{superlow} if $Z' \ltt \ES'$. To be superlow, a ML-random $\DII$ set needs to change considerably,  by a result of Figueira, Hirschfeldt, Miller, Ng and Nies. 
\begin{theorem}
	\cite[Cor.\ 24]{Figueira.Hirschfeldt.ea:nd} Suppose that a \ML\ random set $Z$ is superlow. Then $Z$ is not an $O(2^n)$ change set. 
\end{theorem}
In fact, in \cite[Thm.\ 23]{Figueira.Hirschfeldt.ea:nd} they showed the slightly stronger result that $Z$ is not an $O(h(n) 2^n)$ change set for some order function~$h$. 

In contrast, mere lowness can be achieved with fewer changes: 
\begin{theorem}
	\cite[Thm.\ 11]{Figueira.Hirschfeldt.ea:nd} \label{thm:Fig} Some low \ML\ random set $Z$ is an $o(2^n)$ change set. 
\end{theorem}

We note that the latter result also appears to give some {contrary evidence} to the main principle that $Z$ is computationally less complex if and only if {$Z$ needs more changes}: The set $Z$ constructed in Theorem~\ref{thm:Fig} has a rate of change similar to the one of~$\Omega$, but is low. This suggests that we would need a fine analysis of change bounds in $o(2^n)$ to differentiate between $\Om$ and low random sets. In the the proof of Theorem~\ref{thm:Fig},   the  function $m(k)$ quantifying the ``o'' in $o(2^n)$,  that is,  the minimal $r$ such that for each $n \ge r$, $Z \uhr n$ has at most $\tp{n-k} $ changes, is an $\omega$-c.a. function with $O(4^k)$ increases. In contrast,   $\Om $ only needs $O(2^k)$ increases of its analogous  function.

%
\section*{Act 2: Computably enumerable sets and cost functions} 


\begin{minipage}
	{12cm}

{\it  \emph{The players: } 
 
$\Om$, the King. 

$A$, an abject $\DII$ peasant. 

The king's tax collector. 

\vsps

\emph{The scene: } A village.}

\end{minipage}
\medskip


\n {Book VIII, CHAPTER IV:
How Sir Marhaus came out of Ireland for to ask truage of
Cornwall, or else he would fight therefore.}

\begin{quote}
(...) Then it befell that King Anguish of Ireland sent to Cornwall for his truage\footnote{tribute}, that Cornwall had paid many winters. 
And all that time Cornwall was behind of the truage for seven
years.  And they  gave unto the messenger of
Ireland these words and answer, that they would none pay; and
bade the messenger go unto his King Anguish, and tell him we will
pay him no truage.  (...)  \end{quote}

\subsection*{Cost functions} 

 Suppose the King issues a tax law. This is a computable function $\cost\colon \NN \times \NN \to \QQ^+$ that is nondecreasing in $s$, and nonincreasing in $x$. Consider a computable approximation $(A_s)\sN s$ of a $\DII$ peasant $A$. Suppose that on day $s$, the number $x$ is least such that $A_s(x)$ changes. Then the tax the peasant pays is $\cost(x,s)$. The established terminology for such a   tax law is ``cost function''. Cost functions were used in an ad-hoc way  in \citelow{Kucera.Terwijn:99,Downey.Hirschfeldt.ea:03,Nies:AM}. The general theory was developed    in~\cite[Section 5.3]{Nies:book}, and in more depth in \citelow{Greenberg.Nies:11,Nies:costfunctions}.   
\begin{definition}[\citelow{Nies:book}]
	We say a $\DII$ set $A$ \emph{obeys} a cost function $\cost$ if $A$ has a computable approximation such that the total tax is finite. 
\end{definition}

\vsps

Let $\cost^*(x) = \sup_s \cost(x,s)$. We say that a cost function $\cost$ has the \emph{limit condition} if $\lim_x \cost^*(x)= 0 $. Informally, this is a fair tax law. We show that one can obey each fair tax law without being taxed to death (where death = computable). This result has roots in the work of \Kuc\ and Terwijn~\citelow{Kucera.Terwijn:99} who built an incomputable low-for-random set. Downey et al.\ \citelow{Downey.Hirschfeldt.ea:03} gave a construction like this for the  particular cost function
\begin{equation*}
	\label{eqn:stcf} \cost(x,s) = \sum_{ w =x+1 }^s 2^{-K_s(w)} 
\end{equation*}
in order to build an incomputable $K$-trivial set (see below). In full generality, the construction was first stated in \cite[Thm.\ 5.3.5]{Nies:book}. 

\begin{proposition} \label{prop:Existence}
	Suppose a cost function $\cost$ has the limit condition. Then there is a promptly simple set $A$ obeying $\cost$. 
\end{proposition}

\begin{proof}
	We meet the usual prompt simplicity requirements \bc $PS_e$: \ $ |W_e| =\infty \RRA \exo s \exo x [ x \in W_{e,s } - W_{e,s-1} \lland x \in A_{s}]. $ \ec We define a computable enumeration $\seq{A_s}\sN{s}$ as follows. Let $A_0 = \ES$. At stage $s>0$, for each $e < s$, if $PS_e$ has not been met so far and there is $x \ge 2e$ such that $x \in W_{e, s} - W_{e,s-1} $ and $ \cost(x,s) \le 2^{-e}$, put~$x$ into $A_s$. Declare $PS_e$ to be met.
	
	\vsps
	
	\n Note that $\seq{A_s}\sN{s}$ obeys $ \cost$, since at most one number is put into~$A$ for the sake of each requirement. Thus the total tax the peasant $A$ pays is bounded by $ \sum_e \tp{-e} =2$.
	
	If $W_e$ is infinite, there is an $x\ge 2e$ in $W_e$ such that $ \cost(x,s) \le 2^{-e}$ for all $s>x$, because~$\cost$ satisfies the limit condition. We enumerate such an $x$ into~$A$ at the stage $s>x$ where~$x$ appears in $W_e$, if $PS_e$ has not been met yet by stage~$s$. Thus~$A$ is promptly simple. 
\end{proof}
In the traditional  interpretation (such as \citelow{Soare:87}), being promptly simple would mean that the set changes quickly. So it seems the result says that a set can change quickly in that traditional sense, yet change little in the sense of  the cost function. There is no contradiction because actually, $A$ only has to change quickly \emph{once} for each infinite c.e.\ set $W_e$. This is possible even if the global amount of changes is small. 

We also note that the actual  amount of tax paid is immaterial as long as it is finite: we can always modify the computable approximation   so that the tax becomes arbitrarily small. Thus, a single cost function only distinguishes between sets that change little, and sets that change a lot. Later on,  we will also consider classes of cost function. Jointly obeying each cost function in such a class yields a finer way to gauge the amount of changes.

When studying obedience to a single cost function, we can focus on the c.e.\ sets. 
\begin{proposition}
	[\citelow{Nies:book}, Prop.\ 5.3.6] Suppose a $\DII$ set $A$ obeys a cost function~$\cost$. Then there is a computably enumerable set $D \ge_{T} A$ such that $D$ also obeys $\cost$. If $A$ is $\omega$-c.a., then we can in fact achieve that  $D \ge_{tt} A$. 
\end{proposition}

Recall that $K(x)$ denotes the prefix-free complexity of a string $x$ (see e.g.\ \cite[Ch.\ 2]{Nies:book}, or \citelow{Downey.Hirschfeldt:book}). The Levin-Schnorr theorem characterizes ML-randomness of $Z$ via having an  initial segment complexity $K(Z \uhr n)$ of about~$n$, which is near the upper bound  (see e.g.\ \cite[3.2.9]{Nies:book}). Recall  that  a set $A$ is \emph{$K$-trivial} if for some $b$, $\fao n K(A\uhr n) \le K(n)+b$. Since $K(n) \le 2 \log n  + O(1)$ is the lower bound, this means that $A$ is far from random. 

The following characterizes $K$-triviality among peasants by obedience to the King's tax law $\cost^\Om$,  defined by $\cost^\Om(x,s) = \Om_s - \Om_x$.  This is the amount $\Om$ increases from $x$ to $s$. Note that  $\cost^\Om$  actually depends on a particular computable approximation of  $\Om$ as a left-c.e.\ real.
\begin{theorem}
	[\citelow{Nies:AM},\citelow{Nies:costfunctions}] \label{thm:Char_K} $A$ is $K$-trivial $\LR$ $A$ obeys $\cost^\Om$. 
\end{theorem}

The implication `$\LA$' is not hard. The implication `$\RA$' is also not very hard for a c.e.\ set~$A$, but needs the full power of the so-called golden run method of \citelow{Nies:AM} in the case of a general $\DII$ set $A$. (The proof in \citelow{Nies:AM} was for the cost function $\cost_{\mathcal K}$.) 
\begin{corollary}
	Every $K$-trivial set is Turing below a computably enumerable $K$-trivial set. 
\end{corollary}

Recall the main principle from Act 1: for a ML-random $\DII$ set $Z$. 

\bc $Z$ is computationally less complex $\LLR$ $Z$ needs more changes. \ec \n For c.e.\ sets $A$, we propose a principle that is antipodal to  the one for random $\DII$ sets:

\vsp

\fbox{$A$ is computationally less complex $\LLR$ $A$ obeys stricter cost functions.} 

\vsp

Thus, for c.e.\ sets, being less complex means changing less. We give evidence for this principle, in fact also in the case of left-c.e.\ sets. Similar to Act~1, we proceed from sets of high complexity to sets of low complexity, and see that this complexity matches their changes in the predicted way. We first show that the King pays no taxes. Thereafter we see that peasants get poorer and poorer as they obey stricter and stricter tax laws.

\vsp

\n \emph{Evidence 1.} The left-c.e.\ set $\Omega$ is Turing complete. It obeys no cost function of any reasonable strength, by the following observation. 
\begin{proposition}
	If $\cost$ is a cost function with $\cost(x,s) \ge \tp{-x}$ for all $x,s$, then no \ML\ random $\DII$ set $Z$ obeys $\cost$. 
\end{proposition}
\begin{proof}
	We view $Z_s$ as a binary string. At stage $s>0$, if there is a least~$p$ such that $Z_s(p) \neq Z_{s-1}(p)$, we add the string $Z_s\uhr{p+1}$ to  an effective  list of strings $(\sss_i)\sN i$ as in Definition~\ref{def:MLR}. If $Z$ obeys $\cost$ via $\seq {Z_s}\sN s $, then  $\sum_i \tp{-|\sss_i|} < \infty$. Since  $\sss_i \prec Z$ for infinitely many $i$, $Z$ is not ML-random. 
\end{proof}

\n \emph{Evidence 2.} Bickford and Mills \citelow{Bickford.Mills:nd} studied sets   $A$  such that  $A' \ltt \ES'$.  They called these sets \emph{abject}. They mainly studied this property for c.e.\ sets. Mohrherr \citelow{Mohrherr:86} introduced the term ``superlow'' for this property and also provided results outside the c.e.\ sets.

 The following was first proved using the so-called   golden run method. 
\begin{theorem}
	[\citelow{Nies:AM}] Each $K$-trivial set is superlow. Thus, obeying $\cost^\Om$ implies superlowness. 
\end{theorem}

\n \emph{Evidence 3.} Let $J^A$ be a universal partial computable functional with oracle~$A$. Strong jump traceability, introduced in~\citelow{Figueira.ea:08}, is a lowness property of a set~$A$ saying that the possible values of~$J^A$ are very limited: if $J^A(x) $ is defined at all, then it is contained in a tiny c.e.\ set $T_x$ obtained uniformly from $x$. In ~\citelow{Figueira.ea:08} a c.e.\ but incomputable strongly jump traceable set was built.  Cholak et al.\ \citelow{Cholak.Downey.ea:08} showed among other things that some c.e.\ $K$-trivial set is not strongly jump traceable. In later papers such as \citelow{Downey.Greenberg:nd,Greenberg.Hirschfeldt.ea:nd},  strong jump traceability was studied in great depth.  For general background, see   Section 10.13 of  the excellent book~\citelow{Downey.Hirschfeldt:book}, and also Section 8.5 of~\citelow{Nies:book}. 

A cost function $\cost$ is called \emph{benign}~\citelow{Greenberg.Nies:11} if one can bound computably in~$k$ the number of pairwise disjoint intervals $[x,s)$ with increments $\cost(x,s) \ge \tp{-k}$. For instance, the cost function $\cost^\Om$ used in Theorem~\ref{thm:Char_K} is benign via the bound $k \to 2^k$. Clearly, benignity implies the limit condition.
\begin{theorem}
	[\citelow{Greenberg.Nies:11}] \label{thm:NGbenign} Let $A$ be c.e. Then 
	\bc $A$ is strongly jump traceable $\LR$ $A$ obeys each benign cost function.  \ec
\end{theorem}

Together with Greenberg et al.\ \citelow{Greenberg.Hirschfeldt.ea:nd}, this shows that a c.e.\ set is strongly jump traceable iff it is below each $\omega$-c.a.\ ML-random set. This strengthens the result of Hirschfeldt and Miller from Act~1 that such a set can be incomputable.

Elaborating on Proposition~\ref{prop:Existence}, Franklin and Ng have shown that every benign cost function is obeyed by a strongly prompt set. 
\section*{Act 3: Computably enumerable sets below random $\DII$ sets}
\begin{minipage}
	{12cm} {\it  \emph{The players:}
	
	$Z$, a raundon $\DII$-knight
	
	$A$, a  c.e.\ peasant.
	
	Village  people.
	
	\vsp
	
	\emph{The scene:}
	
	A forest between  village and  castle.}
\end{minipage}

\medskip 

\n {Book VI, Chapter X: How Sir Launcelot rode with a damosel and slew a knight
that distressed all ladies and also a villain that kept a bridge.}

\begin{quote} (...) And so Sir Launcelot and she departed.  And then he rode in a
deep forest two days and more, and had strait lodging.  So on the
third day he rode over a long bridge, and there stert upon him
suddenly a passing foul churl\footnote{archaic: a person of low birth; a peasant (OED)}, and he smote his horse on the nose
that he turned about, and asked him why he rode over that bridge
without his licence.  Why should I not ride this way? said Sir
Launcelot, I may not ride beside.  Thou shalt not choose, said
the churl, and lashed at him with a great club shod with iron. 
Then Sir Launcelot drew his sword and put the stroke aback, and
clave his head unto the paps.  At the end of the bridge was a
fair village, and all the people, men and women, cried on Sir
Launcelot, and said, A worse deed didst thou never for thyself (...)  \end{quote}

\n We now consider the situation that $A \le_T Z$, where $Z$ is a raundon $\DII$~knight, and $A$ is an incomputable c.e.\ peasant. We will see that \vsp

\bc \fbox{the more $Z$ is allowed to change, the less $A$ can change.} \ec

\vsp

The changes of $Z$ are quantified in the sense of initial segments (Act~1). The changes of $A$ are quantified by obeying cost functions (Act~2). This is in line with combining the main principles of these Acts: if $Z$ changes more then  $Z$ is computationally less complex. So the set $A \le_T Z$ is less complex as well, and hence can change less.

The situation above occurs by the following classical theorem of \Kuc\ which says that every raundon $\DII$ knight has  an incomputable  c.e.\ peasant as a  subject. 
\begin{theorem}
	[\citelow{Kucera:86}]  \label{thm:kuc} Let $Z$ be a random $\DII$ set. Then there is a c.e.\ incomputable set $A$ such that $A \le_T Z$. 
\end{theorem}
Greenberg and Nies \citelow{Greenberg.Nies:11} have given a cost function proof of \Kuc's theorem: $A$ is a set obeying a certain cost function $c_Z$ associated with a computable approximation of $Z$. 

Unless $Z \ge_T \Halt$,   the  peasant $A$ in \Kuc{}'s theorem is quite obedient. That is, he is restricted in its amount of possible changes. This follows from a result of Hirschfeldt, Nies, and Stephan. 
\begin{theorem}
	[\citelow{Hirschfeldt.Nies.ea:07}] If $Z$ is Turing incomplete, then a set $A$ as in Theorem~\ref{thm:kuc} is necessarily $K$-trivial. 
\end{theorem}
Depending on the $\DII$ knight $Z$, a c.e.\ peasant subject to $Z$ is can become arbitrarily obedient by a result of  Greenberg et al.~\citelow[Theorem 2.6]{Greenberg.Hirschfeldt.ea:nd}.   
\begin{theorem} \label{thm:P}
	Let $\mathcal P$ be a non-empty $\PPI$ class consisting only of ML-random sets. Let $\cost$ be a cost function with the limit condition. Then there is a $\DII$ set $Z\in \mathcal P$ such that every c.e.\ set $A \le_T Z$ obeys $\cost$. 
\end{theorem} 

This result has a complicated history.  It started with the main result of Greenberg~\citelow{Greenberg:11}.
\begin{theorem} \label{thm:Greenberg09}
 There is a ML-random $\DII$ set $Z$ such that every c.e.\ set $A$ Turing below $Z$ is strongly jump traceable. 
\end{theorem}
Greenberg  built such a  set $Z$ directly in early 2009. Thereafter, \Kuc\ and Nies~\citelow{Kucera.Nies:11} showed that any Demuth random set~$Z$ (see \cite[Section 3.6]{Nies:book}) does the job. (This is another instance of the main principle of this act: if $Z$ is $\DII$, it needs to change a lot in order to be Demuth random. This means that $A$ can only change little.) Greenberg et al.~\citelow{Greenberg.Hirschfeldt.ea:nd}   defined a cost function~$\cost$ such that every c.e.\ set $A$ obeying~$\cost$ is strongly jump traceable. They combined this with  their  Theorem~\ref{thm:P} to obtain  yet another  proof of the    result of  Greenberg. 

In fact, in Theorem~\ref{thm:P}, instead of ML-randomness of $Z$ we can take membership in any non-empty $\PPI$ class by a result of Nies~\citelow{Nies:costfunctions}. 
In that construction, the more restrictive $\cost$, the more $Z$ has to change. If~$\cost$ is benign as defined before Theorem~\ref{thm:NGbenign}, then it is not very restrictive. In this case, the construction makes the set $Z$ $\omega$-c.a. This is predicted by (the contrapositive of) the main principle of this act: if $A$ is allowed more changes, then $Z$ can change less. 
 The extension to $\PPI$ classes shows that in Theorem~\ref{thm:Greenberg09}, randomness of $Z$ can for instance be replaced by PA completeness. 

\bc {\it Exeunt omnes. } \ec 

%
%
%
%
%
%
%
%
%

\begin{thebibliography}{10}

\bibitem{Bickford.Mills:nd}
M.~Bickford and C.F. Mills.
\newblock Lowness properties of r.e.\ sets.
\newblock Preprint, University of Madison, 1982. To appear in J.\ Symb.\ Logic.

\bibitem{Cholak.Downey.ea:08}
P.~Cholak, R.~Downey, and N.~Greenberg.
\newblock Strongly jump-traceability {I}: the computably enumerable case.
\newblock {\em Adv. in Math.}, 217:2045--2074, 2008.

\bibitem{Downey.Greenberg:nd}
R.~Downey and N.~Greenberg.
\newblock Strong jump traceability {II}: $K$-triviality.
\newblock {\em Israel J. Math.} 191: 647-667, 2012.
\newblock In press.

\bibitem{Downey.Hirschfeldt:book}
R.~Downey and D.~Hirschfeldt.
\newblock {\em Algorithmic randomness and complexity}.
\newblock Springer-Verlag, Berlin, 2010.
\newblock 855 pages.

\bibitem{Downey.Hirschfeldt.ea:03}
R.~Downey, D.~Hirschfeldt, A.~Nies, and F.~Stephan.
\newblock Trivial reals.
\newblock In {\em Proceedings of the 7th and 8th Asian Logic Conferences},
  pages 103--131, Singapore, 2003. Singapore University Press.

\bibitem{Figueira.Hirschfeldt.ea:nd}
S.~Figueira, D.~Hirschfeldt, J.~Miller, Selwyn Ng, and A~Nies.
\newblock Counting the changes of random {$\Delta^0_2$} sets.
\newblock In {\em CiE 2010}, pages 1--10, 2010.
\newblock Journal version to appear in J.Logic. Computation.

\bibitem{Figueira.ea:08}
S.~Figueira, A.~Nies, and F.~Stephan.
\newblock Lowness properties and approximations of the jump.
\newblock {\em Ann. Pure Appl. Logic}, 152:51--66, 2008.

\bibitem{Franklin.Ng:10}
J.~Franklin and K.~M. Ng.
\newblock Difference randomness.
\newblock {\em Proceedings of the American Mathematical Society}.
\newblock To appear.

\bibitem{Greenberg:11}
N.~Greenberg.
\newblock A random set which only computes strongly jump-traceable c.e. sets.
\newblock {\em J. Symbolic Logic}, 76(2):700--718, 2011.

\bibitem{Greenberg.Hirschfeldt.ea:nd}
N.~Greenberg, D.~Hirschfeldt, and A.\ Nies.
\newblock Characterizing the strongly jump traceable sets via randomness.
\newblock Adv. Math. 231 (2012), no. 3-4, 2252 -- 2293.



\bibitem{Greenberg.Nies:11}
N.~Greenberg and A.\ Nies.
\newblock Benign cost functions and lowness properties.
\newblock {\em J. Symbolic Logic}, 76:289--312, 2011.

\bibitem{Hirschfeldt.Nies.ea:07}
D.~Hirschfeldt, A.~Nies, and F.~Stephan.
\newblock Using random sets as oracles.
\newblock {\em J. Lond. Math. Soc. (2)}, 75(3):610--622, 2007.

\bibitem{Kucera:86}
A.~Ku{\v{c}}era.
\newblock An alternative, priority-free, solution to {P}ost's problem.
\newblock In {\em Mathematical foundations of computer science, 1986
  (Bratislava, 1986)}, volume 233 of {\em Lecture Notes in Comput. Sci.}, pages
  493--500. Springer, Berlin, 1986.

\bibitem{Kucera.Terwijn:99}
A.~Ku{\v{c}}era and S.~Terwijn.
\newblock Lowness for the class of random sets.
\newblock {\em J. Symbolic Logic}, 64:1396--1402, 1999.

\bibitem{Kurtz:81}
S.~Kurtz.
\newblock {\em Randomness and genericity in the degrees of unsolvability}.
\newblock Ph.{D.} {D}issertation, University of Illinois, Urbana, 1981.

\bibitem{Kucera.Nies:11}
A.~Ku\v{c}era and A~Nies.
\newblock Demuth randomness and computational complexity.
\newblock {\em Ann. Pure Appl. Logic}, 162:504--513, 2011.

\bibitem{Kucera.Nies:12}
Anton\'{\i}n Ku\v{c}era and Andr{\'e} Nies.
\newblock Demuth's path to randomness.
\newblock In {\em Proceedings of the 2012 international conference on
  Theoretical Computer Science: computation, physics and beyond}, WTCS'12,
  pages 159--173, Berlin, Heidelberg, 2012. Springer-Verlag.

\bibitem{Martin-Lof:66}
P.~Martin-L{\"o}f.
\newblock The definition of random sequences.
\newblock {\em Inform. and Control}, 9:602--619, 1966.

\bibitem{Mohrherr:86}
J.~Mohrherr.
\newblock A refinement of low{$_n$} and high{$_n$} for the r.e.\ degrees.
\newblock {\em Z. Math. Logik Grundlag. Math.}, 32(1):5--12, 1986.

\bibitem{Nies:tut}
A.\ Nies.
\newblock Applying randomness to computability.
\newblock University of Auckland preprint based on a series of three lectures
  at the ASL summer meeting, Sofia, 2009. Available at
  \texttt{http://hdl.handle.net/2292/19526}.

\bibitem{Nies:AM}
A.\ Nies.
\newblock Lowness properties and randomness.
\newblock {\em Adv. in Math.}, 197:274--305, 2005.

\bibitem{Nies:book}
A.~Nies.
\newblock {\em Computability and randomness}, volume~51 of {\em Oxford Logic
  Guides}.
\newblock Oxford University Press, Oxford, 2009.

\bibitem{Nies:5questions}
A.\ Nies.
\newblock Studying randomness through computation.
\newblock In {\em Randomness through computation}, pages 207--223. World
  Scientific, 2011.

\bibitem{Nies:costfunctions}
A.\ Nies.
\newblock Calculus of cost functions.
\newblock {I}n preparation, 2012.

\bibitem{Nies.Stephan.ea:05}
A.~Nies, F.~Stephan, and S.~Terwijn.
\newblock Randomness, relativization and {T}uring degrees.
\newblock {\em J. Symbolic Logic}, 70(2):515--535, 2005.

\bibitem{Soare:87}
Robert~I. Soare.
\newblock {\em Recursively Enumerable Sets and Degrees}.
\newblock Perspectives in Mathematical Logic, Omega Series. Springer--Verlag,
  Heidelberg, 1987.

\end{thebibliography}

\end{document}